\documentclass{article}
\usepackage[utf8]{inputenc}
\usepackage{amsmath,amssymb,amsthm,amsfonts,latexsym,graphicx,subfigure,mathtools}

\usepackage{url}
\usepackage{xcolor}

\DeclareMathOperator{\lcm}{lcm}

\theoremstyle{plain}
\newtheorem{theorem}{Theorem}[section]
\newtheorem{lemma}[theorem]{Lemma}
\newtheorem{proposition}[theorem]{Proposition}
\newtheorem{corollary}[theorem]{Corollary}

\theoremstyle{definition}
\newtheorem{definition}[theorem]{Definition}

\theoremstyle{remark}

\usepackage{stmaryrd} 
\usepackage[affil-it]{authblk}
\usepackage{booktabs}   
\usepackage[numbers,sort&compress]{natbib}

\usepackage[margin=3.5cm]{geometry}

\title{Maximal colourings for graphs}

\author[1,2]{Raffaella Mulas\thanks{r.mulas@vu.nl}}
\affil[1]{Vrije Universiteit Amsterdam, Amsterdam, The Netherlands}
\affil[2]{Max Planck Institute for Mathematics in the Sciences, Leipzig, Germany}

\date{}

\allowdisplaybreaks

\begin{document}
	
	\maketitle
	
	\begin{flushright}
		\rightskip=1.8cm\textit{Coming, colours in the air\\
			Oh, everywhere\\
			She comes in colours} \\
		\vspace{.2em}
		\rightskip=.8cm---The Rolling Stones, \textit{She's a Rainbow}
	\end{flushright}
	\vspace{1em}

	\begin{abstract} We consider two different notions of graph colouring, namely, the $t$-periodic colouring for vertices that has been introduced in 1974 by Bondy and Simonovits, and the periodic colouring for oriented edges that has been recently introduced in the context of spectral theory of non-backtracking operators. For each of these two colourings, we introduce the corresponding colouring number which is given by maximising the possible number of colours. We first investigate these two new colouring numbers individually, and we then show that there is a deep relationship between them. \\
		
		\vspace{0.2cm} \noindent {\bf Keywords:} colouring numbers; periodic colouring; circularly partite graphs \\ \end{abstract}

	\section{Introduction}
	\subsection{Historical note}
	While graph theory was born in 1736, when Leonard Euler solved the Königsberg's Seven Bridges Problem \cite{Euler1736}, the history of graph colouring started in 1852, when the South African mathematician and botanist Francis Guthrie formulated the Four Colour Problem \cite{Voloshin2009,Jensen2011,Maritz2012,Wilson2013}. Francis Guthrie noticed that, when colouring a map of the counties of England, one needed at least four distinct colours if two regions sharing a common border could not have the same colour. Moreover, he conjectured (and tried to prove) that four colours were sufficient to colour any map in this way. His brother, Frederick Guthrie, supported him by sharing his work with Augustus De Morgan, of whom he was a student at the time, and De Morgan immediately showed his interest for the problem \cite{Guthrie1880}. On October 23, 1852, De Morgan presented Francis Guthrie's conjecture in a letter to Sir William Rowan Hamilton, in which he wrote:
	\begin{quote}
		\emph{The more I think of it the more evident it seems.}
	\end{quote}But Hamilton replied:
	\begin{quote}
		\emph{I am not likely to attempt your quaternion of colour very soon.}
	\end{quote} De Morgan then tried to get other mathematicians interested in the conjecture, and it eventually became one of the most famous open problems in graph theory and mathematics for more than a century. After several failed attempts in solving the problem, Francis Guthrie's conjecture was proved to be true in 1976, by Kenneth Appel and Wolfgang Haken, with the first major computer-assisted proof in history \cite{Appel1977}.\newline
	
	The Four Colour Theorem can be equivalently described in the language of graph theory as follows. Let $G=(V,E)$ be a \emph{simple graph}, that is, an undirected, unweighted graph without multi-edges and without loops. Two distinct vertices $v$ and $w$ are called \emph{adjacent}, denoted $v\sim w$ or $w\sim v$, if $\{v,w\}\in E$. A \emph{$k$-colouring of the vertices} is a function $c:V\to \{1,\ldots,k\}$, and it is \emph{proper} if $v\sim w$ implies $c(v)\neq c(w)$. 
	The \emph{vertex colouring number} $\chi=\chi(G)$ is the minimum $k$ such that there exists a proper $k$-colouring of the vertices. Moreover, the graph $G$ is called \emph{planar} if it can be embedded in the plane, that is, it can be drawn on the plane in such a way that its edges intersect only at their endpoints.
	\begin{theorem}[Four Colour Theorem, 1976]
		If $G$ is a planar simple graph, then $\chi\leq 4$.
	\end{theorem}
	
	Despite the huge importance of this result, quoting William Thomas Tutte \cite{Tutte1978}, \begin{quote}
		\emph{The Four Colour Theorem is the tip of the iceberg, the thin end of the wedge and the first cuckoo of Spring.}
	\end{quote}
	In fact, the study of the vertex colouring number $\chi$ has shown to be interesting also for several other problems in graph theory, as well as for applications to partitioning problems. Moreover, other notions of colouring have been introduced, and each of them has led to numerous challenging problems, many of which are beautifully summarised in \cite{Tutte1978,Jensen2011}.\newline
	
	For instance, another common notion in the literature is the following. Given a simple graph $G=(V,E)$, a \emph{proper $k$-colouring of the edges} is a function $c:E\rightarrow \{1,\ldots,k\}$ such that, if the distinct edges $e$ and $f$ have one endpoint in common, then $c(e)\neq c(f)$. The \emph{edge colouring number} $\chi^*=\chi^*(G)$ is the minimum $k$ such that there exists a proper $k$-colouring of the edges. In 1964, Vadim Georgievich Vizing \cite{Vizing} proved the following result:
	
	\begin{theorem}[Vizing's Theorem, 1964]
		If $G$ is a simple graph with maximum vertex degree $\Delta$, then $$\chi^*\in \{\Delta,\Delta+1\}.$$
	\end{theorem}
	Graphs with $\chi^*=\Delta$ are said to be of \emph{class one,} and graphs with $\chi^*=\Delta+1$ are said to be of \emph{class two.}

	\subsection{Setting and aim of this paper}
	
	Here we shall focus on two different notions of periodic colouring for a simple graph $G=(V,E)$, as well as on their connection. One of them is the notion of $t$-periodic colouring for vertices that was introduced by John Adrian Bondy and Miklós Simonovits in 1974 \cite{Bondy-Simonovits} (cf.\ Definition 5.4.3 in \cite{Csikvari}). Before defining it, we recall the following
	\begin{definition}
		A \emph{path} of length $\ell\geq 1$ is a sequence\begin{equation*} v_1,\ldots,v_\ell,v_{\ell+1}\end{equation*}of distinct vertices such that $v_j\sim v_{j+1}$ for each $j\in\{1,\ldots,\ell\}$. A \emph{cycle} of length $\ell\geq 3$ is a path where $v_{\ell+1}=v_1$.
	\end{definition}
	
	\begin{definition}A (not necessarily proper) colouring of the vertices of $G$ is \emph{$t$-periodic} if any path of length $t$ in $G$ has endpoints of the same colour.
	\end{definition}
	
	Bondy and Simonovits used this notion as a tool to solve a conjecture of Paul Erdős and to prove that, if $G$ has $N$ nodes and at least $100cN^{1+1/c}$ edges, for a positive constant $c$, then $G$ contains a cycle of length $2\ell$, for every $\ell\in[c,cN^{1/c}]$. In particular, they showed that, for any $t$, if $G$ has average degree $2|E|/|V|\geq 4$, then any $t$-periodic colouring of $G$ has at most $2$ colours (see also  Lemma 5.4.4 in \cite{Csikvari}). However, they did not study properties of the $t$-periodic colouring which were not needed for their main result.\newline
	
	Here we introduce and investigate the following colouring number:
	\begin{definition}
		The \emph{vertex $t$-periodic colouring number} $\chi_t=\chi_t(G)$ is the largest $k$ such that $G$  has a $t$-periodic colouring of the vertices with $k$ colours.
	\end{definition}
	
	Clearly, for each $k\leq \chi_t$, $G$ has a $t$-periodic colouring with $k$ colours. This is why it is appropriate to consider the maximum number of such colours, as opposed to the definitions of $\chi$ and $\chi^*$, where the minimum number of colours is taken.\\
	
	Now, choosing an \emph{orientation} for an edge $\{v,w\}\in E$ means letting one of its endpoints be its \emph{input} and the other one be its \emph{output}. We let $e=[v, w]$ denote the oriented edge whose input is $v$ and whose output is $w$. 
	Moreover, we let \begin{equation*}\mathcal{O}:=\{[v,w], [w,v]\,:\, v\sim w\} \end{equation*}denote the set of oriented edges of $G$.\newline
	
	Before introducing the other notion of periodic colouring that we shall focus on, we recall that a graph is called \emph{$k$-partite} if it admits a proper $k$-colouring of the vertices.
	
	\begin{definition}[\cite{MulasZZ}] Given $k\in\mathbb{N}_{\geq 1}$, $G$ is \emph{circularly $k$--partite} if the set of its oriented edges can be partitioned as $\mathcal{O}=\mathcal{O}_1\sqcup \ldots \sqcup \mathcal{O}_k$, where the sets $\mathcal{O}_j$ are non-empty and satisfy the property that \begin{equation*} [v,w]\in \mathcal{O}_j \Longrightarrow [w,z]\in \mathcal{O}_{j+1}, \, \forall z\sim w\,:\, z\neq v,\end{equation*}where we also let $\mathcal{O}_0:=\mathcal{O}_k$ and $\mathcal{O}_{k+1}:=\mathcal{O}_1$. \end{definition}
	
	In particular, if $G$ is circularly $k$--partite, then the corresponding partition $\mathcal{O}=\mathcal{O}_1\sqcup \ldots \sqcup \mathcal{O}_k$ of the oriented edges naturally defines a $k$-colouring $c:\mathcal{O}\rightarrow \{1,\ldots,k\}$ of the oriented edges by letting\begin{equation*} c([v,w])=j\iff [v,w]\in \mathcal{O}_j. \end{equation*}
	We can also see this colouring as a function $c:\mathcal{O}\rightarrow \mathbb{Z}_k$, where $\mathbb{Z}_k$ is the multiplicative group of integers modulo $k$, and it is \emph{periodic} in the sense that it satisfies
	\begin{equation*} c([v,w])=j \Longrightarrow c([w,z])=j+1, \, \forall z\sim w\,:\, z\neq v.\end{equation*}
	
	Circularly ($k$--)partite graphs have been introduced in \cite{MulasZZ} in the following context. Given a simple graph $G$ as above, its \emph{non-backtracking graph} is the directed graph $\mathcal{G}$ whose vertex set is the set $\mathcal{O}$ of oriented edges of $G$, and such that there is a directed edge from $[v_1,v_2]\in \mathcal{O}$ to $[w_1,w_2]\in \mathcal{O}$ if and only $v_2=w_1$ and $v_1\neq w_2$. The \emph{non-backtracking matrix} of $G$ is defined as the transpose of the adjacency matrix of $\mathcal{G}$. It has been introduced by Ki-ichiro Hashimoto in 1989 \cite{hashimoto1989zeta}, and it has been investigated since then in many areas of graph theory and network science (see for instance \cite{terras2010zeta,bass1992ihara,cooper2009,backhausz2015ramanujan,shrestha2015message,castellano2018relevance,torres2021nonbacktracking,Leo2007,torres2020node,torres2019non,coste2021eigenvalues,glover2021some,pastor2020localization,bordenave2018nonbacktracking,krzakala2013spectral,martin2014localization,Arrigo1,Arrigo2,Gronford1,Mellor1,Godsil16,Huang19}). Recently, in \cite{NB-Laplacian}, Jürgen Jost, Leo Torres and the author of this paper introduced the \emph{non-backtracking Laplacian} of $G$ as the normalized Laplacian $\mathcal{L}$ of $\mathcal{G}$, and they showed that the spectral properties of this new operator are, in some ways, more precise than that of the non-backtracking matrix. In \cite{NB-Laplacian} it is shown, in particular, that the eigenvalues of the non-backtracking Laplacian $\mathcal{L}$ are contained in the complex disc $D(1,1)$. Moreover, the \emph{spectral gap from $1$} is defined as the smallest distance $\varepsilon$ between the eigenvalues of $\mathcal{L}$ and the point $1$ on the real line, and a sharp lower bound for $\varepsilon$ is provided, but no upper bound is given in \cite{NB-Laplacian}. As a continuation of this work, in \cite{MulasZZ}, Dong Zhang, Giulio Zucal and the author of this paper proved additional spectral properties of the non-backtracking Laplacian, including a sharp upper bound for $\varepsilon$. And in order to prove this upper bound (Theorem 5.1 in \cite{MulasZZ}), they had to introduce the above class of circularly partite graphs.\newline
	
	Here we introduce also the following colouring number:
	
	\begin{definition}The \emph{oriented edge periodic colouring number} $\chi^o=\chi^o(G)$ is the largest $k$ such that $G$ is circularly $k$--partite.\end{definition}
	
	We shall see why, as for $\chi_t$, also in this case it is natural to consider the maximum number of colours, and we shall investigate various properties of $\chi^o$. Moreover, we shall see that there is a deep connection between $\chi_t$ and $\chi^o$, despite the fact that their definitions are quite different, and they arose in two unrelated contexts.\\
	
	
	Throughout the paper, we shall consider simple graphs, and we shall assume that all graphs are \emph{connected.}
	
	\subsection*{Structure of the paper}
	
	In Section \ref{section:characterisation} we
	give a characterisation of circularly $k$-partite graphs that does not involve the set of oriented edges. Then, in Section \ref{section:chi^p} we prove various properties of the oriented edge periodic colouring number $\chi^o$, and similarly, in Section \ref{Section:Bondy} we investigate properties of the vertex $t$-periodic colouring number $\chi_t$. Finally, in Section \ref{Section:both}, we study the relationship between $\chi^o$ and $\chi_t$.


	\section{Characterisation of circularly partite graphs}\label{section:characterisation}

	As pointed out already in \cite{MulasZZ}, all graphs are circularly $1$-partite, and a graph is circularly $2$-partite if and only if it is bipartite. An example of a circularly $3$-partite graph is given in Figure \ref{fig:3partite} below.\newline
	
	\begin{figure}[h]
		\centering
		\includegraphics[width=9cm]{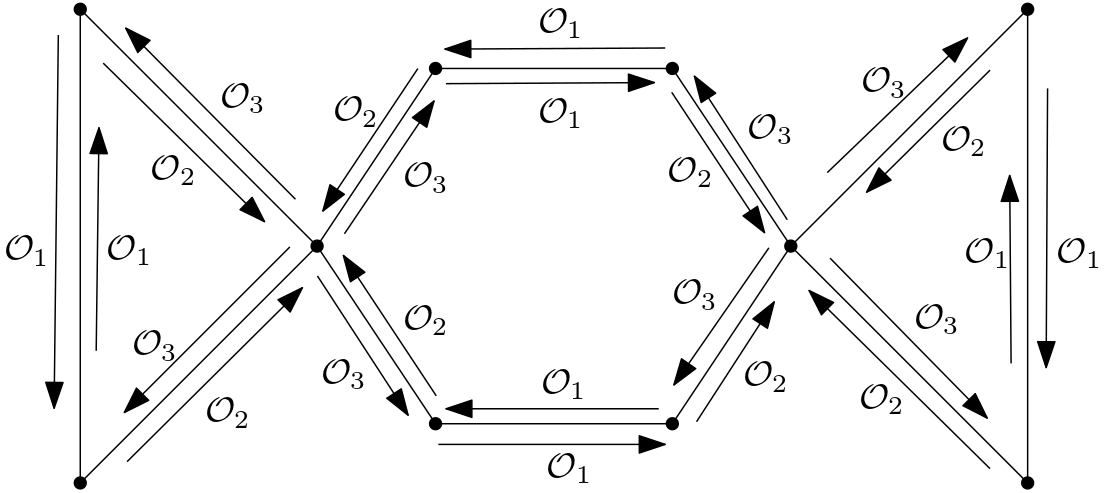}
		\caption{A circularly $3$-partite graph.}
		\label{fig:3partite}
	\end{figure}

	We shall now give a characterisation of all circularly $k$-partite graphs that does not require looking at the set of oriented edges. Before, we need to define extended star graphs, as well as the diameter of a graph.
	
	\begin{definition}
		A graph $G$ is an \emph{extended star graph} if it is a tree (that is, it has no cycles) and it has exactly one vertex of degree $\geq 3$.
	\end{definition}
	
	For instance, star graphs with at least $3$ edges are extended star graphs.

	\begin{definition}
		The \emph{diameter} of a graph is the length of the shortest path between the most distanced nodes. 
	\end{definition}
	
	We can characterise all circularly $k$-partite graphs as follows.

	\begin{theorem}\label{thm:part}   \begin{enumerate} \item The path graph of length $M$ is circularly $k$--partite for each $k\in\{1,\ldots,2M\}$. 
			\item An extended star graph of diameter $M$ is circularly $k$--partite for each $k\in\{1,\ldots,M\}$.
			\item For $k\geq 1$, a graph $G$ which is neither a path nor an extended star graph is circularly $k$--partite if and only if the following two conditions hold:
			\begin{itemize}
				\item The length of each cycle is a multiple of $k$, and
				\item The length $\ell$ of any path connecting two distinct vertices of degree $\geq 3$ is such that $2\ell$ is a multiple of $k$.
			\end{itemize}
		\end{enumerate}
	\end{theorem}
	\begin{proof}
		The first two claims are straightforward, and so is the third claim for $k=1$, since all graphs are circularly $1$-partite and the two conditions are satisfied by all graphs for $k=1$. For $k=2$, we know that a graph is circularly $2$-partite if and only if it is bipartite, therefore if and only if all cycles have even length. Together with the fact that $2\ell$ is always a multiple of $2$, this proves the claim for $k=2$.\newline
		
		Now, fix $k\geq 3$ and let $G$ be a circularly $k$-partite graph which is not a path. Clearly, the length of each cycle in $G$ must be a multiple of $k$. To see that also the second condition holds, assume that there is a path of length $\ell$ between two distinct vertices $v$ and $w$ of degree $\geq 3$. Assume also, without loss of generality, that one of the oriented edges that have $v$ as an input belongs to $\mathcal{O}_1$. Then, since $\deg v\geq 3$, all oriented edges that have $v$ as an input must be in $\mathcal{O}_1$, and all oriented edges that have $v$ as an output must be in $\mathcal{O}_k$. This leads us to the situation depicted in Figure \ref{fig:thm-partite}, where we also let $\mathcal{O}_{k+r}:=\mathcal{O}_r$ for any $r$. In particular, we have an oriented edge in $\mathcal{O}_{2\ell}$ that has $v$ as output, but we know that all oriented edges that have $v$ as output belong to $\mathcal{O}_k$. This implies that $2\ell$ must be a multiple of $k$. The fact that, vice versa, a graph satisfying conditions 1 and 2 is circularly $k$-partite, is also clear from Figure \ref{fig:thm-partite}.
		
		\begin{figure}[h]
			\centering
			\includegraphics[width=9cm]{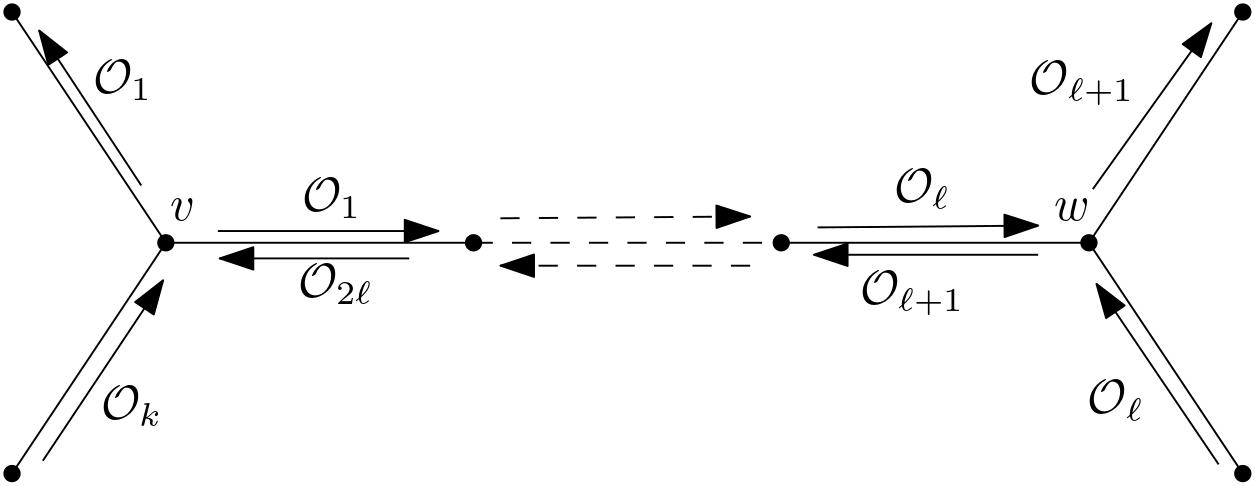}
			\caption{An illustration of the proof of Theorem \ref{thm:part}.}
			\label{fig:thm-partite}
		\end{figure}
		
	\end{proof}

	By Theorem \ref{thm:part}, it is straightforward to see, for example, that a Mickey Mouse graph (Figure \ref{fig:mickey}) with both ears of length $2k$, the face of length $5k$ and distance $k$ between the ears is circularly $k$-partite.\newline
	
	\begin{figure}[h]
		\centering
		\includegraphics[width=4cm]{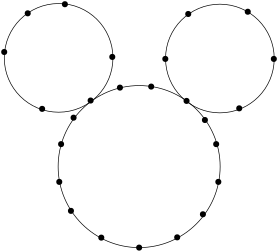}
		\caption{A Mickey Mouse graph is circularly $k$-partite.}
		\label{fig:mickey}
	\end{figure}

	Moreover, two immediate consequences of Theorem \ref{thm:part} are the following corollaries.
	
	\begin{corollary}\label{cor:ll}
		Let $G$ be a circularly $k$-partite graph which is not a path graph. If we know  which set $\mathcal{O}_j$ a given oriented edge $[v,w]$ belongs to, then by induction we can also infer to which set each oriented edge belongs to.
	\end{corollary}
	
	\begin{corollary}\label{cor:dividers}
		If a graph is circularly $k$-partite, then it is also circularly $\ell$-partite, for each number $\ell$ that divides $k$. 
	\end{corollary}
	
	Hence, it is natural to ask what the largest $k$ such that a graph is circularly $k$-partite is.

	\section{Oriented edge periodic colouring number}\label{section:chi^p}

	In this section we focus on the \emph{oriented edge periodic colouring number} $\chi^o=\chi^o(G)$, which we defined as the largest $k$ such that $G$ is circularly $k$--partite.\\

	We start by considering the following examples.
	\begin{itemize}
		\item[--] For the path graph of length $M$, $\chi^o=2M$.
		\item[--] For star graphs, $\chi^o=2$. \item[--] More generally, for an extended star graph of diameter $M$, $\chi^o=M$.
		\item[--] For the cycle graph on $N$ nodes, $\chi^o=N$.
		\item[--] The \emph{petal graph} with $\ell \geq 1$ petals of length $k \geq 3$ is the graph given by $\ell$ cycle graphs of length $k$, all having a common central vertex. Clearly, for a petal graph whose petals have length $k$, we have $\chi^o=k$. 
	\end{itemize}
	
	We can now improve Corollary \ref{cor:dividers} as follows.

	\begin{theorem}\label{thm:dividers}
		If $G$ is neither a path nor an extended star graph, then for every $k\geq 1$,
		\begin{equation*}
		G\text{ is circularly $k$--partite } \iff k \text{ divides }\chi^o.
		\end{equation*}
	\end{theorem}
	\begin{proof}
		The $(\Leftarrow)$ follows from Corollary \ref{cor:dividers}. Now, assume that $G$ is circularly $k$--partite. Then, by definition of oriented edge periodic colouring number, we must have $k\leq \chi^o$. Assume also, by contradiction, that $k$ does not divide $\chi^o$. Then, $k>1$ and, by Theorem \ref{thm:part},
		\begin{itemize}
			\item The length of each cycle is a multiple of both $\chi^o$ and $k$, and
			\item The length $\ell$ of any path connecting two distinct vertices of degree $\geq 3$ is such that $2\ell$ is a multiple of both $\chi^o$ and $k$.
		\end{itemize}This implies that
		\begin{itemize}
			\item The length of each cycle is a multiple of $\lcm(k,\chi^o)$, and 
			\item The length $\ell$ of any path connecting two distinct vertices of degree $\geq 3$ is such that $2\ell$ is a multiple of $\lcm(k,\chi^o)$.
		\end{itemize}Therefore, again by Theorem \ref{thm:part}, $G$ is circularly $\lcm(k,\chi^o)$--partite. But since we are assuming that $k\leq \chi^o$ does not divide $\chi^o$, we have that $\lcm(k,\chi^o)>\chi^o$, which leads to a contradiction.
	\end{proof}
	
	
	\begin{lemma}\label{lemma:qr}
		Let $q,r\in \mathbb{N}_{\geq 3}$ be two relatively prime numbers. If $G$ has at least one cycle of length $q$ and at least one cycle of length $r$, then $\chi^o=1$.
	\end{lemma}\begin{proof}
		It follows from Theorem \ref{thm:part}.
	\end{proof}
	
	For example, the graphs in Figure \ref{fig:chip14} only differ by one edge subdivison. However, by Lemma \ref{lemma:qr}, the first graph is such that $\chi^o=1$, since the inner cycles have length $3$ while the outer cycle has length $4$. On the other hand, by Theorem \ref{thm:part}, the second graph is such that $\chi^o=4$.\newline
	
	\begin{figure}
		\centering
		\includegraphics[width=10cm]{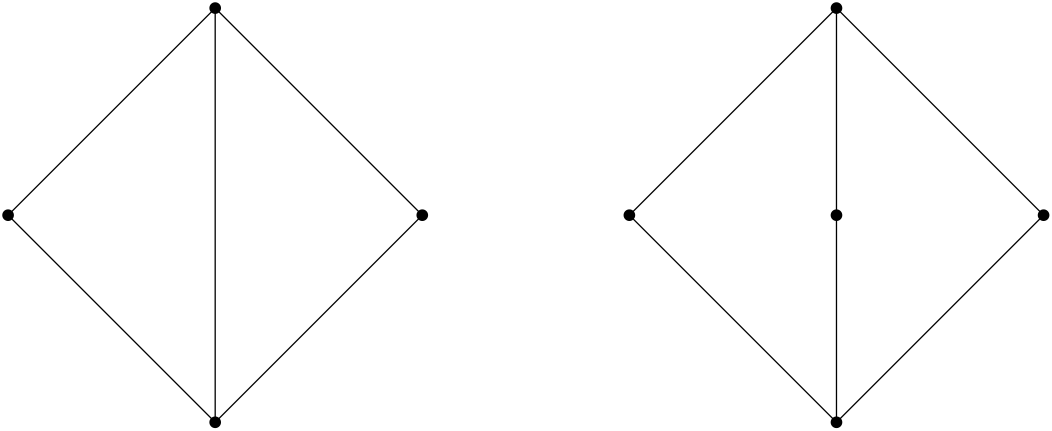}
		\caption{A graph with $\chi^o=1$ and a graph with $\chi^o=4$.}
		\label{fig:chip14}
	\end{figure}

	An immediate consequence of Lemma \ref{lemma:qr} is the following
	
	\begin{corollary}
		Let $K_N$ denote the complete graph on $N$ nodes, and let $G$ be a graph on $N$ nodes. If $K_4$ is a subgraph of $G$, then $\chi^o(G)=1$. In particular, if $N\geq 4$, then $\chi^o(K_N)=1$.
	\end{corollary}
	
	Hence, for $N\geq 4$, the complete graph $K_N$ has the smallest possible oriented edge periodic colouring number, since $\chi^o(K_N)=1$, and the largest possible vertex colouring number, since $\chi(K_N)=N$.
	
	\begin{proposition}\label{prop:regular}Let $d\geq 3$. If $G$ is $d$-regular, then $\chi^o\in\{1,2\}$.
	\end{proposition}\begin{proof}Clearly, by the assumptions, $G$ is neither a path nor an extended star graph. Hence, by Theorem \ref{thm:part} and since all vertices have degree $d\geq 3$, the length $\ell$ of any path connecting two distinct vertices is such that $2\ell$ is a multiple of $\chi^o$. This is also true for the distance $\ell=1$ between two adjacent vertices. Therefore, $2$ must be a multiple of $\chi^o$, implying that $\chi^o\leq 2$.\end{proof}
	
	As a consequence of Proposition \ref{prop:regular}, cycle graphs are the only regular graphs for which $\chi^o>2$.\\

	Next, we shall evaluate the vertex and edge colouring number for all graphs that have $\chi^o>1$.\newline
	
	Observe that, for an extended star graph of diameter $M$, $\chi^o=M$ and $\chi=2$. The colouring number of all other graphs that have $\chi^o>1$ is given in the next theorem.
	
	\begin{theorem}\label{thm:chi}
		If $G$ is not an extended star graph, then
		\begin{itemize}
			\item[(a)] $\chi^o$ is even if and only if $\chi=2$.
			\item[(b)] If $\chi^o>1$ is odd, then $\chi=3$.
		\end{itemize}

	\end{theorem}
	\begin{proof}
		\begin{itemize}
			\item[(a)] 
			We know that 
			\begin{equation*}
			\chi=2 \iff G\text{ is bipartite} \iff G\text{ is circularly $2$-partite}.
			\end{equation*}By Theorem \ref{thm:dividers}, if $G$ is neither a path nor an extended star graph, this happens if and only if $2$ divides $\chi^o$, that is, $\chi^o$ is even. If $G$ is a path graph of length $M$, the claim follows from the fact that $\chi=2$ and $\chi^o=2M$ is always even.
			\item[(b)] 
			By (a), $G$ is not bipartite, therefore $\chi\geq 3$. Hence, if we show that there exists a proper $3$-colouring of the vertices we are done, as this implies that $\chi \leq 3$. Since $G$ is neither a path (as it is not bipartite) nor an extended star graph (by assumption), it is described by the third point in Theorem \ref{thm:part}, for $k=\chi^o$. Therefore, the length $\ell$ of any path connecting two distinct vertices of degree $\geq 3$ must be such that $2\ell$ is a multiple of $\chi^o$. But since $\chi^o$ is odd, this implies that $\ell$ must be a multiple of $\chi^o$. Hence, since $\chi^o>1$, also $\ell>1$. In particular, two vertices of degree $\geq 3$ cannot be adjacent, as their distance must be greater than $1$. Therefore, we can colour all vertices of degree $\geq 3$ with the same colour. Since all other vertices have degree $1$ or $2$, it is enough to colour them with at most two additional colours to obtain a proper $3$-colouring of the entire vertex set. This proves the claim.
		\end{itemize}
	\end{proof}

	The inverse of Theorem \ref{thm:chi}(b) does not hold. For instance, let $G$ be the graph given by a cycle of length $3$ attached with a cycle of length $4$ by a common central vertex. Then, $\chi=3$ and, by Lemma \ref{lemma:qr}, $\chi^o=1$.\newline
	
	We now consider the edge colouring number. We shall prove that all graphs that have $\chi^o>1$ are of class one.
	
	\begin{theorem}
		For any graph with maximum degree $\Delta$,
		\begin{equation*}
		\chi^o>1 \Rightarrow \chi^*=\Delta.
		\end{equation*}
	\end{theorem}
	\begin{proof} If $G$ is bipartite, then $\chi^*=\Delta$ by Kőnig's line colouring Theorem \cite{Konig}. Hence, we can now assume that $G$ is not an extended star graph, $\Delta\geq 3$ and, by Theorem \ref{thm:chi}(a), we can assume that $\chi^o\geq 3$ is odd. In this case, as shown in the proof of Theorem \ref{thm:chi}(b), if $v$ and $w$ are two distinct vertices of degree $\geq 3$, then their distance $d$ must be a multiple of $\chi^o$. Hence, in particular, $d\geq 3$. This allows us to make a proper $\Delta$-colouring of the edges, as follows:
		\begin{enumerate}
			\item We can first colour all edges that are incident to vertices of degree $\geq 3$ with colours $\{1,2,3,\ldots,\Delta\}$, so that two such edges that share a common vertex have different colours.
			\item The endpoints of the remaining edges all have degree $1$ or $2$. Clearly, we can colour such edges with colours $\{1,2,3\}$ so that any two edges in $G$ that meet at a common vertex have different colours.
		\end{enumerate}Therefore, $\chi^*\leq \Delta$, and by Vizing's Theorem, this implies that $\chi^*= \Delta$.
	\end{proof}
	
	\section{Vertex $t$-periodic colouring}\label{Section:Bondy}
	
	Recall that, in the Introduction, we gave the following definitions: 
	
	\begin{definition}A (not necessarily proper) colouring of the vertices of $G$ is \emph{$t$-periodic} if any path of length $t$ in $G$ has endpoints of the same colour.
	\end{definition}

	\begin{definition}
		The \emph{vertex $t$-periodic colouring number} $\chi_t=\chi_t(G)$ is the largest $k$ such that $G$  has a $t$-periodic colouring of the vertices with $k$ colours.
	\end{definition}
	

	Clearly, $\chi_t\geq 1$ for any $t$. We shall assume that $t$ is small enough that, if we know the colours of any given $t$ vertices $v_1\sim \cdots\sim v_t$ that form a path, then by periodically repeating these colours we know how to colour the entire graph. With this assumption, we have that $\chi_t\leq t$, since the colours of $v_1,\ldots, v_t$ (which are at most $t$) can be used to colour the entire graph.\\



	Now, we know that a graph $G$ is bipartite if and only if $\chi=2$. Moreover, in Theorem \ref{thm:chi} we have shown that, if $G$ is not an extended star graph, then this is true if and only if $\chi^o$ is even. Similarly, we can prove the following
	\begin{proposition}\label{prop:chi2}
		$G$ is bipartite if and only if $\chi_2=2$.
	\end{proposition}
	\begin{proof}
		If $V=V_1\sqcup V_2$ is a partition of the vertex set that makes $G$ bipartite, then we can clearly colour $V_1$ with one colour and $V_2$ with another colour to obtain a $2$-periodic colouring. Together with the fact that $\chi_2\leq 2$, this implies that $\chi_2=2$.\newline
		Assume now that $\chi_2=2$, and let $V=V_1\sqcup V_2$ be the partition of the vertex set which is given by the colours of a $2$-periodic colouring. Assume also, by contradiction, that there exists an edge $v\sim w$ within $V_1$. Then, for each $z\in V_2$, we cannot have $v\sim z$, because otherwise $w$ would have the same colour as $z$, and similarly we cannot have $w\sim z$. Hence, in particular, $v$ and $w$ can only be connected to vertices within $V_1$. By induction, this implies that $V_1$ and $V_2$ are not joined by any edge, which is a contradiction since we are considering connected graphs.
	\end{proof}
	
	We now consider some examples. 
	
	\begin{itemize}
		\item[-] Clearly, $\chi_1=1$ for every graph.
		\item[-] For the path on $N$ nodes and any $t\in\{1,\ldots,N\}$, it is easy to check that $\chi_t=t$.
		\item[-] For the complete graph on $N$ nodes, $\chi_t=1$ for all $t\in \{1,\ldots,N-1\}$.
		\item[-]  In \cite{Bondy-Simonovits} and \cite[Lemma 5.4.4]{Csikvari} it is shown that, for any $t$, if $G$ has average degree $2|E|/|V|\geq 4$, then any $t$-periodic colouring of $G$ has at most $2$ colours, implying that $\chi_t\leq 2$ in this case.
	\end{itemize}
	
	\begin{theorem}\label{thm:cycle}
		Let $G=C_N$ be the cycle graph on $N$ nodes, and let $t\in\{1,\ldots,N\}$. Then,
		\begin{equation*}
		\chi_t=\gcd(t,N).
		\end{equation*}
	\end{theorem}
	\begin{proof}
		Let $\gamma:=\gcd(t,N)$. The fact that $\chi_t\geq \gamma$ is clear by considering a proper $t$-periodic $\gamma$-colouring of the vertices. Hence, it is enough to show that $\chi_t\leq \gamma$. Let
		\begin{equation*}
		v_0\sim v_1\sim \ldots\sim v_{N-1}
		\end{equation*}be the vertices of $G$, and let $c:V\rightarrow \{1,\ldots,\chi_t\}$ be a $t$-periodic colouring. Moreover, write $t=\gamma e$ and $N=\gamma f$, so that $\gcd(e,f)=1$. If we see the indices of all the vertices as elements of the multiplicative group $\mathbb{Z}_N$ of integers modulo $N$, then by definition of $t$-periodic colouring, we have that 
		\begin{equation*}
		c(v_i)=c(v_{\alpha t +i}), \quad  \forall \alpha\in\mathbb{Z}_N \text{ and }\forall i\in \{0,\ldots,N-1\}.
		\end{equation*} Therefore, if we show that for each $i\in \{0,\ldots,N-1\}$ there are $f$ distinct elements of the form $\alpha t+i$ (for $\alpha \in\mathbb{Z}_N$) in $\mathbb{Z}_N$,
		we can infer that there are at least $f$ distinct vertices that have colour $c(v_i)$. In particular, this would imply that 
		\begin{equation*}
		\chi_t\leq  \frac{N}{f}=\gamma.
		\end{equation*} Without loss of generality, we can consider $i=0$ and ask how many distinct elements of the form $\alpha t$ (for $\alpha \in\mathbb{Z}_N$) are in $\mathbb{Z}_N$. We have
		\begin{align*}
		\alpha t\equiv \alpha' t \mod{N} &\iff N\vert (\alpha-\alpha')t\\
		&\iff \gamma f\vert (\alpha-\alpha')\gamma e\\
		&\iff f\vert (\alpha-\alpha')e\\
		&\iff f\vert (\alpha-\alpha')\\
		&\iff \alpha \equiv \alpha' \mod{f}.
		\end{align*}Since $f$ is the number of distinct elements $\alpha\in\mathbb{Z}_f$, this proves the claim.
	\end{proof}
	
	An immediate consequence is the following \begin{corollary}\label{cor:cycle}  If $G$ has a cycle of length $\ell$, then for $t\in\{1,\ldots,\ell\}$,\begin{equation*} \chi_t\leq \gcd(t,\ell).\end{equation*}\end{corollary}\begin{proof}  Clearly, if $C_{\ell}$ is a subgraph of $G$, then any $t$-periodic colouring of $G$ is also a $t$-periodic colouring of $C_{\ell}$. And by Theorem \ref{thm:cycle}, any $t$-periodic colouring of $C_{\ell}$ must be $\leq \gcd(t,\ell)$. \end{proof}
	
	We now consider graphs with minimum degree $\geq 2$ which are not cycle graphs.
	
	\begin{definition}
		The \emph{girth} $g(G)$ of a graph $G$ is the length of a shortest cycle contained in $G$.
	\end{definition}

	\begin{theorem}\label{thm:tau}
		Let $G$ be a graph with minimum degree $\geq 2$ which is not a cycle graph. Fix $t\geq 3$ which is not larger than $g(G)$, and let $\tau:=\bigl\lfloor \frac{t}{2} \bigr\rfloor$. Then, any $t$-periodic colouring must be of the form
		\begin{align}
		&c_0,c_1, \ldots,c_{\tau},\ldots,c_1 \quad \text{(for $t$ even), or}\label{eq:t-even}\\
		&c_0,c_1, \ldots, c_{\tau},c_{\tau},\ldots,c_1 \quad \text{(for $t$ odd).}\label{eq:t-odd}
		\end{align}
		Hence, in particular, 
		\begin{equation*} \chi_t\leq \tau+1,
		\end{equation*}
		with equality if and only if the $\tau+1$ colours in \eqref{eq:t-even} or \eqref{eq:t-odd} can be chosen to be all distinct.
	\end{theorem}
	\begin{proof}
		Since we are assuming that $G$ is a graph with minimum degree $\geq 2$ which is not a cycle graph, and since we are assuming that $t$ is not larger than $g(G)$, one can check that $G$ contains a subgraph of the form $C_\ell \sqcup P_\tau$ which is given by a cycle on $\ell\geq t$ nodes attached to a path of length $\tau$ at exactly one vertex $v_0$. 
		Let 
		\begin{equation*}
		v_0\sim v_1\sim \ldots\sim v_{\ell-1}
		\end{equation*}denote the vertices within the cycle $C_\ell$, and consider the indices of these vertices as elements of $\mathbb{Z}_{\ell}$. Let also
		\begin{equation*}
		v_0\sim w_1\sim \ldots \sim w_\tau
		\end{equation*}be a path of length $\tau$ such that the vertices $w_i$'s are not in $C_{\ell}$.

		\begin{figure}[h]
			\centering
			\includegraphics[width=4cm]{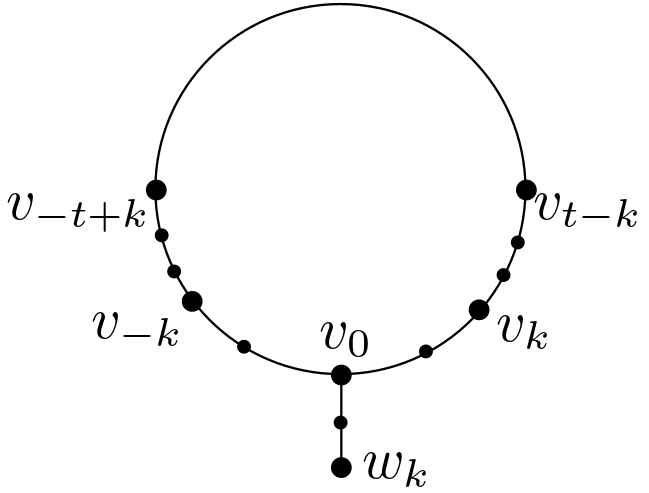}
			\caption{The vertices $w_k$, $v_k$, $v_{-k}$, $v_{t-k}$ and $v_{-t+k}$ all have the same colour.}
			\label{fig:CPk}
		\end{figure}

		For $k\leq \tau$, any $t$-periodic colouring $c$ must satisfy
		(Figure \ref{fig:CPk})
		\begin{equation*}
		c(w_k)=c(v_{t-k})=c(v_{-k})
		\end{equation*} and
		\begin{equation*}
		c(w_k)=c(v_{-t+k})=c(v_k).
		\end{equation*}
		The condition $$c(w_k)=c(v_k)=c(v_{t-k}) \quad \forall k\in \{1,\ldots, \tau\}$$ implies that
		\begin{align*}
		& c(w_1)=c(v_1)=c(v_{t-1}),\\
		& c(w_2)=c(v_2)=c(v_{t-2}),\\
		&\ldots\\
		& c(w_{\tau})=c(v_\tau)=c(v_{t-\tau}).
		\end{align*}
		
		Therefore, if $t=2\tau$ is even, then the $t$-periodic colouring has to be of the form
		\begin{equation*}
		c(v_0),c(w_1), \ldots, c(w_{\tau}),\ldots,c(w_1).
		\end{equation*}If $t=2\tau+1$ is odd, then the $t$-periodic colouring has to be of the form
		\begin{equation*}
		c(v_0),c(w_1), \ldots,c(w_{\tau}),c(w_{\tau}),\ldots,c(w_1).
		\end{equation*}
		This proves the claim.
		
	\end{proof}

	\section{Vertex $t$-periodic colouring and oriented edge periodic colouring}\label{Section:both}
	
	In this concluding section, we show that there is a deep relationship between the vertex $t$-periodic colouring and the oriented edge periodic colouring. We start with the following

	\begin{proposition}\label{prop:cyclet}
		If $G$ is the cycle graph on $N$ nodes and $t\in \{1,\ldots,N\}$, then 
		\begin{equation*} \chi_t= t \iff G \text{ is circularly $t$-partite}.
		\end{equation*}
	\end{proposition}
	\begin{proof}
		By Theorem \ref{thm:cycle}, $\chi_t= t$ if and only if $\gcd(t,N)=t$, therefore if and only if $t$ divides $N$. By Theorem \ref{thm:part}, this happens if and only if $G$ is circularly $t$-partite.
	\end{proof}

	\begin{theorem}\label{thm:tau+1}
		Let $G$ be a graph with minimum degree $\geq 2$ which is not a cycle graph. Fix $t\geq 3$ which is not larger than $g(G)$, and let $\tau:=\bigl\lfloor \frac{t}{2} \bigr\rfloor$. Then, 
		\begin{equation*}  \chi_t= \tau+1\iff G \text{ is circularly $t$-partite}.
		\end{equation*}
	\end{theorem}
	
	\begin{proof}
		If $t$ is even, then by Theorem \ref{thm:part}, $G$ is circularly $t$-partite if and only if
		\begin{itemize}
			\item The length of each cycle is a multiple of $t$, and
			\item The length of any path connecting two distinct vertices of degree $\geq 3$ is a multiple of $\tau$.
		\end{itemize}
		As a consequence one can check that, in this case, the $t$-periodic colouring in \eqref{eq:t-even} can be made with $\tau+1$ distinct colours, if each vertex of degree $\geq 3$ has either colour $c_0$ or colour $c_\tau$. \newline
		
		Similarly, if $t$ is odd, then again by Theorem \ref{thm:part}, $G$ is circularly $t$-partite if and only if\begin{itemize}
			\item The length of each cycle is a multiple of $t$, and
			\item The length of any path connecting two distinct vertices of degree $\geq 3$ is a multiple of $t$.
		\end{itemize} As a consequence one can check that, in this case, the $t$-periodic colouring in \eqref{eq:t-odd} can be made with $\tau+1$ distinct colours, if all vertices of degree $\geq 3$ have colour $c_0$.\newline
		
		Therefore, we have shown one direction. Assume now that $\chi_t= \tau+1$, and let $c:V\rightarrow \{1,\ldots,\tau+1\}$ be a $t$-periodic colouring. Then, by Theorem \ref{thm:tau}, $c$ must be of the form
		\begin{align*}
		&c_0,c_1, \ldots,c_{\tau},\ldots,c_1 \quad \text{(for $t$ even), or}\\
		&c_0,c_1, \ldots, c_{\tau},c_{\tau},\ldots,c_1 \quad \text{(for $t$ odd),}
		\end{align*}
		where $c_0,\ldots,c_{\tau}$ are $\tau+1$ distinct colours. Hence, clearly, the length of each cycle must be a multiple of $t$. This shows the first condition for circularly $t$-partite graphs.\newline
		
		Now, if $v_0$ is a vertex of degree $\geq 3$, then we can write
		\begin{align*}
		&v_0\sim v_1\sim \ldots \sim v_{t-1},\\
		&v_0\sim v_{-1}\sim \ldots \sim v_{-t+1},\\
		&v_0\sim w_1,
		\end{align*}where:
		\begin{itemize}
			\item The vertices $v_1$, $v_{-1}$ and $w_1$ are all distinct;
			\item The vertices $v_i$, for $i\in\{0,1,\ldots,t-1,-1,\ldots,-t+1\}$ might possibly have indices in $\mathbb{Z}_\ell$, for some $\ell\geq t$. By definition of $t$-periodic colouring,
			\begin{equation*}
			c(w_1)=c(v_{t-1})=c(v_1)
			\end{equation*}and
			\begin{equation*}
			c(w_1)=c(v_{-t+1})=c(v_{-1}),
			\end{equation*}implying that all neighbors of $v_0$ must have the same colour. But since the colours $c_0,\ldots,c_\tau$ are all distinct, this is only possible if either $t$ is even and $v_0$ has colour $c_0$ or $c_\tau$, or $t$ is odd and $v_0$ has colour $c_0$. Hence, this implies that the length of any path connection two distinct vertices of degree $\geq 3$ must be a multiple of $\tau$ (for $t$ even), or a multiple of $t$ (for $t$ odd). That is, $G$ must be circularly $t$-partite.
		\end{itemize}
	\end{proof}

	Hence, summarising the relationship between the $t$-periodic colouring for vertices and the periodic colouring for oriented edges, we have that:
	\begin{enumerate}
		\item Every graph admits a $1$-periodic $1$-colouring of the vertices, as well as a periodic $1$-colouring of the oriented edges.
		\item  A graph is bipartite if and only if it admits a $2$-periodic $2$-colouring of the vertices (by Proposition \ref{prop:chi2}), and if and only if it admits a periodic $2$-colouring of the oriented edges (by Theorem \ref{thm:chi}).
		\item  If $G$ is the cycle graph on $N$ nodes and $t\in \{1,\ldots,N\}$, then $\chi_t\leq t$, and
		\begin{equation*} \chi_t= t \iff G \text{ is circularly $t$-partite} \iff t \text{ divides }N,
		\end{equation*} by Proposition \ref{prop:cyclet} and Theorem \ref{thm:dividers}.
		\item   Let $G$ be a graph with minimum degree $\geq 2$ which is not a cycle graph. Fix $t\geq 3$ which is not larger than $g(G)$, and let $\tau:=\bigl\lfloor \frac{t}{2} \bigr\rfloor$. Then $\chi_t\leq \tau+1$, and
		\begin{equation*}  \chi_t= \tau+1\iff G \text{ is circularly $t$-partite},
		\end{equation*} by Theorem \ref{thm:tau} and Theorem \ref{thm:tau+1}.
	\end{enumerate}

	\subsection*{Conflict of interest}
	The author declares that she has no conflict of interest.

	\subsection*{Acknowledgments}
	Most of these ideas have been developed during one of the author's visits to the Alfréd Rényi Institute of Mathematics in Budapest. She is grateful to Ágnes Backhausz, Peter Csikvari, Péter Frenkel, László Lovász and Giulio Zucal for the inspiring discussions that took place during her visit. She is also grateful to Conor Finn, Joseph Lizier and (again) Giulio Zucal for the discussions that took place at the Max Planck Institute for Mathematics in the Sciences. She is grateful to the anonymous referees for the comments and suggestions that have improved the paper.

	\bibliographystyle{plain} 
	
	\bibliography{Colouring}
	
\end{document}